\documentclass[11pt]{amsart}
\usepackage{amssymb, amstext, amscd, amsmath}
\usepackage{mathtools, xypic, color, dsfont, rotating, paralist}
\usepackage[all]{xy}

\makeatletter
\def\@cite#1#2{{\m@th\upshape\bfseries%
[{#1\if@tempswa{\m@th\upshape\mdseries, #2}\fi}]}}
\makeatother
\theoremstyle{plain}
\newtheorem{theorem}{Theorem}[section]
\newtheorem{corollary}[theorem]{Corollary}
\newtheorem{proposition}[theorem]{Proposition}
\newtheorem{lemma}[theorem]{Lemma}

\theoremstyle{definition}

\newtheorem{remark}[theorem]{Remark}

\newtheorem*{acknow}{Acknowledgements}
\theoremstyle{remark}

\mathtoolsset{centercolon}
\newcommand{\bC}{{\mathds{C}}}
\newcommand{\bF}{{\mathds{F}}}

\newcommand{\bN}{{\mathds{N}}}

\newcommand{\bR}{{\mathds{R}}}
\newcommand{\bT}{{\mathds{T}}}
\newcommand{\bZ}{{\mathds{Z}}}
  
  \newcommand{\B}{{\mathcal{B}}}

  \newcommand{\E}{{\mathcal{E}}}

\renewcommand{\O}{{\mathcal{O}}}

  \newcommand{\T}{{\mathcal{T}}}
  \newcommand{\U}{{\mathcal{U}}}

  \newcommand{\X}{{\mathcal{X}}}

\def\al{\alpha}
\def\be{\beta}
\def\ga{\gamma}
\def\de{\delta}

\def\ka{\kappa}

\def\la{\lambda}
\def\si{\sigma}

\def\vpi{\varphi}

\newcommand{\rC}{{\mathrm{C}}}

\newcommand{\Bs}{{\mathbf{s}}}

\newcommand{\Bv}{{\mathbf{v}}}

\newcommand{\foral}{\text{ for all }}
\newcommand{\qand}{\quad\text{and}\quad}
\newcommand{\qif}{\quad\text{if}\quad}
\newcommand{\qfor}{\quad\text{for}\ }

\newcommand{\ca}{\mathrm{C}^*}

\newcommand{\ol}{\overline}
\newcommand{\wt}{\widetilde}
\newcommand{\wh}{\widehat}

\newcommand{\Aut}{\operatorname{Aut}}

\newcommand{\diag}{\operatorname{diag}}

\newcommand{\End}{\operatorname{End}}

\newcommand{\ev}{\operatorname{ev}}

\newcommand{\id}{{\operatorname{id}}}
\newcommand{\im}{{\operatorname{Im}}}

\newcommand{\mt}{\emptyset}

\newcommand{\spn}{\operatorname{span}}

\newcommand{\sca}[1]{\left\langle#1\right\rangle} 
\newcommand{\nor}[1]{\left\Vert #1\right\Vert} 

\long\def\symbolfootnote[#1]#2{\begingroup%
\def\thefootnote{\fnsymbol{footnote}}\footnote[#1]{#2}\endgroup}

\addtocontents{toc}{\protect\setcounter{tocdepth}{1}}

\begin{document}

\title[KMS states on Pimsner algebras associated with C*-dynamics]{KMS states on Pimsner algebras associated with C*-dynamical systems}

\author[E.T.A. Kakariadis]{Evgenios T.A. Kakariadis}
\address{School of Mathematics and Statistics\\ Newcastle University\\ Newcastle upon Tyne\\ NE1 7RU\\ UK}
\email{evgenios.kakariadis@ncl.ac.uk}


\maketitle

\begin{abstract}
We examine the theory of the KMS states on Pimsner algebras arising from multivariable unital C*-dynamical systems.
As an application we show that Pimsner algebras of piecewise conjugate classical systems attain the same KMS states, even though it is still an open problem whether or not they are $*$-isomorphic.
\end{abstract}

\section{Introduction}

In the most part of this paper we investigate the structure of the KMS states on Pimsner algebras associated with C*-dynamical systems.
The motivation for this work is two-fold.
First of all we are inspired by the growing interest on the structure of the KMS states that involves: the Cuntz algebra \cite{OlePed78}, Cuntz-Krieger algebras \cite{EFW84}, Hecke algebras \cite{BosCon95}, C*-algebras associated with subshifts \cite{MWM98}, Pimsner algebras \cite{LacNes04}, the Toeplitz algebra of $\bN \rtimes \bN^\times$ \cite{LacNes11, LacRae10}, C*-algebras of dilation matrices \cite{LRR11}, C*-algebras of self-similar actions \cite{LRRW13}, topological dynamics \cite{AHR14, Tho11, Tho12}, higher-rank graphs \cite{HLRS14}, and Nica-Pimsner algebras \cite{Kak14-2}.
Secondly we are interested in analyzing further equivalence relations on multivariable classical systems \cite{DavKat11, KakKat12}.
Our motivation relies on the interaction of that theory with other fields of mathematics.
In a series of papers Cornelissen \cite{Cor12} and Cornelissen and Marcolli \cite{CorMar11, CorMar13} illustrate such a strong link by inserting and applying the seminal work on piecewise conjugacy of Davidson and Katsoulis \cite{DavKat11} into number theory and the reconstruction of graphs.
It remains an open problem whether piecewise conjugacy implies unitary equivalence for multivariable classical systems, and thus whether these two notions coincide.
A positive answer would allow the use of the flexible local features of piecewise conjugacy in order to overpass the inconvenient attributes of the global aspect of unitary equivalence.
Since unitary equivalent systems admit $*$-isomorphic Pimsner algebras, we use the KMS states to test this question.

Given $d$ $*$-endomorphisms $\al_i$ of a C*-algebra $A$ there is a Toeplitz-Pimsner algebra $\T(A,\al)$ and a Cuntz-Pimsner algebra $\O(A,\al)$ that one can form.
Both of them admit an action $\si$ of $\bR$ induced by the gauge action.
We show that there is a phase transition at the inverse temperature $\be = \log d$ in the following sense:
(a) for $\be \in (-\infty, \log d)$ there are no KMS states (Proposition \ref{P: cond KMS});
(b) for $\be \in (\log d, +\infty)$ there is an affine weak*-homeomorphism from the simplex of the tracial states on $A$ onto the $(\si,\be)$-KMS states on $\T(A,\al)$ (Theorem \ref{T: KMS TP}).
Moreover there is an affine weak*-homeomorphism from the simplex of the tracial states (resp. states) on $A$ onto the simplex of the KMS${}_\infty$ states (resp. ground states) on $\T(A,\al)$ (Proposition \ref{P: ground TP} and Proposition \ref{P: infty TP}).
By using that $\O(A,\al)$ is a quotient of $\T(A,\al)$ we show that the same results pass over for the tracial states on $A$ that vanish on the ideal $(\cap_{i=1}^d \ker \al_i)^\perp$ (Theorem \ref{T: KMS CP} and Corollary \ref{C: ground/infty CP}).
However when $(A,\al)$ is injective then there are no $(\si,\be)$-KMS states on $\O(A,\al)$ for all $\be \neq \log d$ (Proposition \ref{P: inj CP}).

Different phenomena appear for the critical temperature $\be = \log d$, making a unification hard to achieve.
Our analysis in this case does not yield a parametrization.
For example if $A = C(X)$ and $d=1$ then $\O(A,\id) = C(X \times \bT)$ admits plenty of $(\si,0)$-KMS states, i.e. tracial states.
For $A = C(X)$ and $d>1$ we have that any state on $A$ induces a $(\si,\log d)$-KMS state on $\O(A,\id)$; in particular if $A = \bC$ and $d>1$ then $\O(\bC,\id) = \O_d$ has a unique $(\si,\log d)$-KMS state \cite{OlePed78, Eva80, BEK80}.
We show that there are (finite dimensional) cases where the $(\si,\log d)$-KMS state is unique and other (finite dimensional) cases where there are more than one (Remark \ref{R: comm logd CP}).
In addition we give a sufficient and necessary condition for $\O(A,\al)$ to have KMS states at $\log d$ when $(A,\al)$ is injective and $d >1$ (Proposition \ref{P: affine log d}).
This approach gives a direct way for obtaining $(\si,\log d)$-KMS states on $\O(A,\al)$ for injective classical systems (Remark \ref{R: comm logd CP}).
This is achieved by using the operator that averages over the actions $\al_i$, and it is inspired by the Perron-Frobenius type theory of Matsumoto, Watatani and Yoshida \cite{MWM98}.
The case where $d=1$ is reduced to finding tracial states on a usual C*-crossed product (Proposition \ref{P: tracial crossed}).
A key tool in this approach is the tail adding techniques of the author with Katsoulis \cite{Kak11-1, KakKat11}.

The analysis of Laca and Neshveyev \cite{LacNes04} on Pimsner algebras makes also use of a Perron-Frobenius type operator.
Moreover it concerns actions of $\bR$ beyond the ones inherited from the gauge action.
Their approach is rather illuminating but it seems hard to be adopted in our specific example to provide a parametrization of the KMS states.
It is the recent works of Laca and Raeburn \cite{LacRae10}, and Laca, Raeburn, Ramagge and Whittaker \cite{LRRW13} that supply us with the efficient tools and algorithms for our purposes.
We inform the reader that \cite{LacNes04} follows Pimsner's setting \cite{Pim97}, and the results about Cuntz-Pimsner algebras require a small reformulation to agree with the modern language of C*-correspondences set by Katsura \cite{Kat04}.
In \cite{Kat04} Katsura introduced a version of a Cuntz-Pimsner algebra that effectively treats non-injective C*-correspondences.
Katsura's Cuntz-Pimsner algebra has become a rather important C*-algebra.
A key feature is that it is the smallest relative Cuntz-Pimsner algebra that contains an isometric copy of the C*-correspondence \cite{Kak13-2}.
This information is crucial when associating the KMS states of $\O(A,\al)$ to those of $\T(A,\al)$ in the proof of Theorem \ref{T: KMS CP}.
Indeed the quotient mapping $q \colon \T(A,\al) \to \O(A,\al)$ does not interfere with the restriction of a KMS state on $A$, and so with it being tracial on $A$.

The tensor algebra of a C*-dynamical system in the sense of Muhly and Solel \cite{MuhSol98} consists of the lower triangular non-involutive part of $\T(A,\al)$.
The author and Katsoulis \cite{KakKat12} have shown that isometric isomorphisms of the tensor algebras is a complete invariant for unitary equivalent systems when $A=C(X)$.
On the other hand Davidson and Katsoulis \cite{DavKat11} show that piecewise conjugacy is implied by algebraic isomorphism of the tensor algebras.
The converse is known to hold in specific cases (e.g. $n=2,3$) and gives stronger isomorphisms by isometric maps \cite{DavKat11}.
The combination of \cite{DavKat11} and \cite{KakKat12} suggests that the converse in full generality could be achieved by proving that unitary equivalence and piecewise conjugacy are equivalent.
We note here that unitary equivalent systems are piecewise conjugate (e.g. Proposition \ref{P: un eq pc}), and it is the converse of this fact that concerns us.

It is immediate that unitary equivalent systems admit $*$-isomorphic Pimsner algebras.
Therefore an examination of the aforementioned problem can be conducted within the class of C*-algebras; in particular by analyzing the invariants of C*-algebras.
The form of the KMS states that we obtain is accommodating for this task.
As an application we show that Pimsner algebras of piecewise conjugate systems admit the same theory of KMS states (Corollary \ref{C: comm same KMS}).
This gives evidence that unitary equivalence and piecewise conjugacy may be equivalent.
However the same conclusion is derived for any equivalence relation that respects the orbit of a point in the sense of Lemma \ref{L: same sum comm}.
Even more this conclusion can be reformulated to cover non-classical systems as well: if $(A,\al)$ and $(C,\ga)$ are dynamical systems of the same multiplicity and there is a $*$-isomorphism $\wh{\phi} \colon A \to C$ such that
\[
\{\tau \al_w(a) \mid w\in \bF_+^d, |w| = m\} = \{\tau \wh{\phi}^{-1}\ga_w \wh{\phi}(a) \mid w\in \bF_+^d, |w| =m\}
\]
for all $a\in A$ and $m \in \bZ_+$, then $\T(A,\al)$ and $\T(C,\ga)$ (resp. $\O(A,\al)$ and $\O(C,\ga)$) admit the same KMS states in the strong sense of Corollary \ref{C: comm same KMS}.

\section{Preliminaries}

\subsection{Kubo-Martin-Schwinger states}

Let $\si \colon \bR \to \Aut(A)$ be an action on a C*-algebra $A$.
Then there exists a norm-dense $\si$-invariant $*$-subalgebra $A_{\textup{an}}$ of $A$ with the following property: for every $a\in A_{\textup{an}}$ the function $\bR \ni t \mapsto \si_t(a) \in A$ is analytically continued to an entire function $\bC \ni z \mapsto \si_z(a) \in A$ (see \cite[Proposition 2.5.22]{BraRob87}).
A state $\psi$ of $A$ is called a \emph{$(\si,\be)$-KMS state} if it satisfies
\[
\psi(a b) = \psi(b \si_{i\be}(a)),
\]
for all $a,b$ in a norm-dense $\si$-invariant $*$-subalgebra of $A_{\text{an}}$.
If $\be=0$ or if the action is trivial then a KMS state is a tracial state on $A$.

The KMS condition follows as an equivalent for the existence of particular continuous functions. More precisely, for $\be>0$ let 
\[
D = \{ z \in \bC \mid 0 < \im(z) < \be\}.
\]
Then a state $\psi$ is a $(\si,\be)$-KMS state if and only if for any pair $a,b \in A$ there exists a complex function $F_{a,b}$ that is analytic on $D$ and continuous (hence bounded) on $\ol{D}$ such that
\[
F_{a,b}(t) = \psi(a \si_t(b)) \qand F_{a,b}(t + i \be) = \psi(\si_t(b)a),
\]
for all $t\in \bR$ (see \cite[Proposition 5.3.7]{BraRob97}).

A state $\psi$ of a C*-algebra $A$ is called a \emph{ground state} if the function $z \mapsto \psi(a \si_{z}(b))$ is bounded on $\{z \in \bC \mid \text{Im}z >0\}$ for all $a,b$ inside a dense analytic subset $A_{\text{an}}$ of $A$.
A state $\psi$ of $\T(A,\al)$ is called a \emph{KMS$_\infty$ state} if it is the w*-limit of $(\si,\be)$-KMS states as $\be \longrightarrow \infty$.
The reader should be aware of the fact that this distinction is not apparent in \cite{BraRob87,BraRob97} and is coined in \cite{LacRae10, LRRW13}.

\subsection{C*-dynamical systems}\label{Ss: dyn sys}

A \emph{C*-dynamical system $(A,\al)$ of multiplicity $d$} consists of $d$ $*$-endomorphisms $\al_i$ of a C*-algebra $A$.
In particular when $A = C_0(X)$ for a locally compact Hausdorff space $X$ then each $\al_i$ is identified by a proper continuous map $\si_i \colon X \to X$.
In this case we will call $(A,\al) \equiv (X,\si)$ a \emph{multivariable classical system}.
We will say that $(A,\al)$ is \emph{unital} if $A$ has a unit and every $\al_i$ is unital for $i=1, \dots, d$.
We will say that $(A,\al)$ is \emph{non-degenerate} if all the $\al_i$ are non-degenerate for $i=1, \dots, d$.
We will say that $(A,\al)$ is \emph{injective} if $\cap_{i=1}^d \ker\al_i = (0)$.
The required ideal in $A$ for the covariant representations is coined by Katsura in \cite[Definition 3.2]{Kat04}.
It is defined by
\[
J_{(A,\al)} := (\cap_{i=1}^d \ker \al_i)^\perp.
\]
It is immediate that $(A,\al)$ is injective if and only if $J_{(A,\al)} = A$.

There is a C*-correspondence construction associated with $(A,\al)$.
It goes back to the work of Davidson and Katsoulis on classical systems \cite{DavKat11}, exploited further by Davidson and Roydor \cite{DavRoy10}, and by the author with Katsoulis \cite{KakKat11, KakKat12}.
It is not necessary for our purposes to review the whole theory of C*-correspondences.
Instead we will give the appropriate definitions in terms of $(A,\al)$, and provide brief comments to facilitate comparisons.

We denote by $\T(A,\al)$ the universal C*-algebra generated by the formal monomials $\Bv_w a$ with $a\in A$ and $w \in \bF_+^d$, such that $[\Bv_1, \dots, \Bv_d]$ is a row isometry and $a \Bv_i = \Bv_i \al_i(a)$ for all $a\in A$ and $i=1, \dots, d$.
We will refer to $\T(A,\al)$ as the Toeplitz-Pimsner algebra of $(A,\al)$.
The existence of $\T(A,\al)$ can be checked by using a general result of Loring \cite[Theorem 3.1.1]{Lor97}, or the usual C*-folklore technique; alternatively see \cite[Chapter 2]{DFK14}.
In particular there is a well known construction that gives (resp. faithful) representations of $\T(A,\al)$.
Let $\pi \colon A \to \B(H)$ be a (resp. faithful) representation of $A$.
On the Hilbert space $K = H \otimes \ell^2(\bF_+^d)$ let the orbit representation
\[
\wt{\pi}(a) = \diag\{\pi \al_{\ol{w}}(a) \mid w \in \bF_+^d\}, \foral a \in A,
\]
where $\ol{w} = \ol{i_m \dots i_1} = i_1 \dots i_m$ is the reversed word of $w = i_m \dots i_1$.
By defining $V_i \xi \otimes e_w = \xi \otimes e_{iw}$ we form a row isometry $[V_1, \dots, V_d]$.
Then the mappings
\[
a \mapsto \wt{\pi}(a) \qand \Bv_i \mapsto V_i
\]
lift to a (resp. faithful) representation $V \times \wt\pi$ of $\T(A,\al)$ onto the C*-algebra $\ca(\wt{\pi},V)$ generated by $V_w \wt{\pi}(a)$ for all $a\in A$ and $w \in \bF_+^d$.
Due to the relations on the generators we check that
\[
\ca(\wt\pi,V) = \ol{\spn} \{ V_\mu a V_\nu^* \mid a \in A, \mu, \nu \in \bF_+^d\}.
\]
Consequently, injectivity of $V \times \wt\pi$ when $\pi$ is faithful follows by a usual gauge-invariant-uniqueness-theorem argument as in \cite[Theorem 6.2]{Kat04}.
A simplified proof of \cite[Theorem 6.2]{Kat04} can be found in \cite{Kak13-2}.
From now on let us fix such a family $(\wt{\pi}, \{V_i\}_{i=1}^d)$ where $\pi$ is faithful.
Since $A$ embeds isometrically in $\T(A,\al)$ we will drop the use of $\wt{\pi}$.

\begin{remark}
Let us briefly describe the C*-correspondence associated with $(A,\al)$.
Full details are left to the interested reader.
Let the C*-correspondence $\E = \sum_{i=1}^d \E_i$ such that $\E_i = A$ is a C*-correspondence over $A$ endowed with the operations
\[
\sca{\xi,\eta} = \xi^*\eta \qand a \cdot \xi \cdot b = \al_i(x)\xi b
\]
for all $a,b \in A$ and $\xi, \eta \in \E_i = A$.
Then the Toeplitz-Pimsner algebra $\T(\E)$ is $*$-isomorphic to $\T(A,\al)$.
The reader is addressed to \cite[Theorem 2.10]{DavKat11} modulo \cite[Remark 8.4]{KakSha15}.
In the case of unital systems \cite[Theorem 2.10]{DavKat11} suffices.
These ideas extend to cover non-classical dynamical systems in general.
Alternatively one may check that the representation $V \times \wt\pi$ for a faithful $\pi\colon A \to \B(H)$ satisfies the required conditions of \cite[Proposition 6.1]{Kat04}, and thus gives a faithful representation of $\T(\E)$ as well.
Another suggestion is to use \cite[Theorem 2.1]{FowRae99}.
This line of reasoning requires some more elaboration on the construction of the C*-correspondence $\E$ associated with $(A,\al)$.

We just mention that the Cuntz-Pimsner algebra $\O(\E)$ in the sense of Katsura \cite[Definition 3.5]{Kat04} is $*$-isomorphic to the C*-algebra $\O(A,\al)$ we describe below.
The reader is addressed also to \cite{KakKat12}.
We remark that we make use of the gauge invariant uniqueness theorems in their general form as proven by Katsura \cite[Theorem 6.2 and Thoerem 6.4]{Kat04}.
The reader is addressed also to \cite{Kak13-2} for an alternative proof and an overview on the subject.
\end{remark}

We denote by $\O(A,\al)$ the Cuntz-Pimsner algebra related to $(A,\al)$, i.e. the quotient of $\T(A,\al)$ by the ideal generated by the differences
\[
a - \sum_{i=1}^d V_i \al_i(a) V_i^* = a(I - \sum_{i=1}^d V_i V_i^*), \foral a \in J_{(A,\al)}.
\]
We denote by $q \colon \T(A,\al) \to \O(A,\al)$ the quotient map.
Since $A$ embeds isometrically inside $\O(A,\al)$ we will make the identification $a \equiv q(a)$ for all $a\in A$.
Moreover we denote by $S_{\mu} = q(V_{\mu})$ for all $\mu \in \bF_+^d$.
Therefore $\O(A,\al)$ is the universal C*-algebra generated by the formal monomials $\Bs_w a$ with $a\in A$ and $w \in \bF_+^d$, such that $[\Bs_1, \dots, \Bs_d]$ is a row isometry, $a \Bs_i = \Bs_i \al_i(a)$ for all $a\in A$ and $i=1, \dots, d$, and
\[
a(I - \sum_{i=1}^d \Bs_i \Bs_i^*) = 0, \foral a \in J_{(A,\al)}.
\]
In particular, when $(A,\al)$ is injective then $\O(A,\al)$ is the universal C*-algebra generated by the $\Bs_i a$ such that $[\Bs_1, \dots, \Bs_d]$ is a row unitary and $a \Bs_i = \Bs_i \al_i(a)$ for all $a\in A$ and $i=1, \dots, d$.

When $(A,\al)$ is unital then $1 \in A$ is also the unit for $\T(A,\al)$ and $\O(A,\al)$.
If $(A,\al)$ is not unital then we form the unitization $(A^{(1)},\al^{(1)})$ so that $\al_i^{(1)}(a + \la) = \al_i(a) + \la$.
We denote by $A^{(1)} = A + \bC$ the unitization of $A$, even when $A$ has a unit.
We make the convention that $A^{(1)} = A$ only when $(A,\al)$ is unital.
Since the unitization is an extension of the original system we get that $\T(A,\al)$ can be realized as a C*-subalgebra of $\T(A^{(1)},\al^{(1)})$.
Indeed consider a Fock representation $(\wt{\pi},\{V_i\}_{i=1}^d)$ of $(A^{(1)},\al^{(1)})$, take $(\wt{\pi}|_A,\{V_i\}_{i=1}^d)$ and use the gauge invariant uniqueness theorem for the Toeplitz-Pimsner algebras.
On the other hand, when $A$ is not unital then $A$ is an essential ideal of $A^{(1)}$, and when $A$ is unital then $A^{(1)} = A \oplus \bC$.
Therefore we obtain that $J_{(A,\al)} = J_{(A^{(1)},\al^{(1)})}$.
By the gauge invariant uniqueness theorems we have that $\O(A,\al)$ is a C*-subalgebra of $\O(A^{(1)},\al^{(1)})$ respectively.

There is a well-known tail adding technique due to Muhly and Tomforde \cite{MuhTom04} for dilating a non-injective C*-correspondence $\E$ to an injective one $\X$ such that the Cuntz-Pimsner algebra $\O(\E)$ is a full corner of the Cuntz-Pimsner algebra $\O(\X)$.
This process follows a pattern similar to the tail adding technique for eliminating sources in graphs. 
The construction of Muhly and Tomforde \cite{MuhTom04} has found several applications, however it has also some limitations.
The author and Katsoulis have observed in \cite[Proposition 3.12]{KakKat11} that it does not preserve classes in the following sense.
There is an example of a non-injective C*-dynamical system whose dilation, as constructed in \cite{MuhTom04}, is not a C*-correspondence of an injective C*-dynamical system.
In order to tackle this problem the author and Katsoulis \cite{KakKat11} introduced a technique that takes into consideration a variety of different tails \cite[Theorem 3.10]{KakKat11}.
A careful choice of the appropriate tail then does the trick.
Among others this technique generalizes the one variable technique of the author \cite[Theorem 4.8]{Kak11-1} and fixes an error in the original technique of Davidson and Roydor \cite[Section 4]{DavRoy10} for classical multivariable systems, see \cite[Example 4.3]{KakKat11} and \cite[Example 4.4]{KakKat11} respectively.

Once more we will not require the full theory of C*-correspondences to review the dilation in the case of C*-dynamical systems.
More details may be found in \cite[Section 4]{KakKat11}.
Let $(A,\al)$ be a non-injective C*-dynamical system.
Let the direct sum C*-algebra
\[
C = C_0 \oplus (\oplus_{w \in \bF_+^d} C_w)
\]
where $C_0 = A$ and $C_w = A/J_{(A,\al)}$ for all $w\in \bF_+^d$.
Note that $C_0$ is not to be mistaken with $C_\mt$.
We then define the $*$-endomorphisms $\ga_i \colon C \to C$ such that
\[
\big(\ga_1(a,(c_w))\big)_\mu = \begin{cases} \al_1(a) & \qif \mu = 0, \\ q_J(a) & \qif \mu = \mt, \\ c_{1^{n}} & \qif \mu = 1^{n+1}, n \geq 0, \\ c_w & \qif \mu = 1 w, \mu \neq 1^n, n \geq 1, \\ 0 & \text{ otherwise}, \end{cases}
\]
where $q_J \colon A \to A/ J_{(A,\al)}$, and
\[
\big(\ga_i(a,(c_w))\big)_\mu = \begin{cases} \al_i(a) & \qif \mu = 0,\\ c_w & \qif \mu = i w, \\ 0 & \text{ otherwise}, \end{cases}
\]
for $i=2,\dots, d$.
For $d=2$ this is depicted in the following diagram
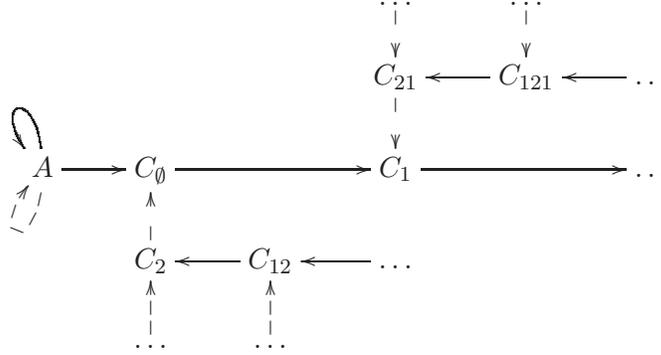
\[
\xymatrix@R=1.7em{
& & &\dots \ar@{-->}[d] & \dots \ar@{-->}[d] & \\
& & & C_{21} \ar@{-->}[d] &  C_{121} \ar[l] & \dots \ar[l] \\
A \ar@(u,ul)[] \ar@{-->}@(d,dl)[] \ar[r] & C_\mt \ar[rr] & & C_1 \ar[rr] & & \dots \\
& C_2 \ar@{-->}[u] & C_{12} \ar[l] & \dots \ar[l] & & \\
& \dots \ar@{-->}[u] & \dots \ar@{-->}[u] & & &
}
\]
where the solid arrows represent $\ga_1$ and the broken arrows $\ga_2$, with the understanding that we don't write the cases where the elements are sent to zero.
When $(A,\al)$ is non-degenerate then $(C,\ga)$ is non-degenerate as well; however when $(A,\al)$ is unital then $(C,\ga)$ is no longer unital.
We call $(C,\ga)$ \emph{the injective dilation of $(A,\al)$}.
(Here the $*$-endomorphisms $\{\ga_i\}_{i=1}^d$ should not be confused with the gauge action $\{\ga_z\}_{z \in \bT}$ on the Pimsner algebras.)

In particular when $d=1$ we write $\al_1 = \al$ and $\ga_1 = \ga$.
In this case we have that $C = A \oplus c_0(A/\ker\al^\perp)$ and
\[
\ga(a, (c_n)) = (\al(a), a + \ker\al^\perp, (c_n)).
\]
In turn we may extend the injective system $(C,\ga)$ of multiplicity $1$ to the direct limit automorphic system $(\wt{C},\wt{\ga})$ given by the diagram
\begin{align*}
\xymatrix@C=3em{
C \ar[r]^{\ga} \ar[d]^{\al} &
C \ar[r]^{\ga} \ar[d]^{\al} &
C \ar[r]^{\ga} \ar[d]^{\al} &
\cdots \ar[r] &
\wt{C} \ar[d]^{\wt{\ga}} \\
C \ar[r]^{\ga} &
C \ar[r]^{\ga} &
C \ar[r]^{\ga} &
\cdots \ar[r] &
\wt{C}
}
\end{align*}
Then we have that $\O(C,\ga) = \O(\wt{C},\wt{\ga}) = \wt{C} \rtimes_{\wt{\ga}} \bZ$.
The system $(\wt{C},\wt{\ga})$ is called \emph{the automorphic dilation of $(A,\al)$}.
The interested reader is addressed to \cite[Section 4]{Kak11-1} for the pertinent details.

Finally we mention that we will write $|\mu| = n$ for the length of a word $\mu = \mu_n \dots \mu_1 \in \bF_+^d$.

\section{KMS states on Pimsner algebras}

\subsection{The Toeplitz-Pimsner algebra}

Given the gauge action $\{\ga_z\}_{z \in \bT}$ of $\T(A,\al)$ we let $\si \colon \bR \to \Aut(\T(A,\al))$ be the action with $\si_t = \ga_{e^{it}}$.
The linear span of the monomials $V_\mu a V_\nu^*$, with $\mu,\nu \in \bF_+^d$ and $a \in A$, is dense in $\T(A,\al)$, and
\[
\si_t (V_\mu a V_\nu^*) = e^{i(|\mu| - |\nu|)t} V_\mu a V_\nu^*.
\]
Since the function $t \mapsto \si_t(V_\mu a V_\nu^*)$ extends to the entire function
\[
z \mapsto e^{i(|\mu| - |\nu|)z} V_\mu a V_\nu^*,
\]
then the $(\si,\be)$-KMS condition for a state $\psi$ is equivalent to
\begin{align*}
\psi(V_\mu a V_\nu^* \cdot V_\ka b V_\la^*)
& =
\psi(V_\ka b V_\la^* \cdot \si_{i\be} (V_\mu a V_\nu^*)) \\
& =
e^{-(|\mu| - |\nu|)\be} \psi(V_\ka b V_\la^* \cdot V_\mu a V_\nu^*)
\end{align*}
for all $a,b \in A$ and $\mu, \nu, \ka, \la \in \bF_+^d$.
Let us begin with the following conditions on the form of KMS states.

\begin{proposition}\label{P: cond KMS}
Let $(A,\al)$ be a unital C*-dynamical system with multiplicity $d$.

\begin{inparaenum}
\item Suppose that $\be < \log d$.
Then $\T(A,\al)$ has no $(\si,\be)$-KMS states.

\item Suppose that $\be \geq \log d$ and $\be > 0$.
Then $\psi$ is a $(\si,\be)$-KMS state for $\T(A,\al)$ if and only if $\psi|_A$ is tracial and
\[
\psi(V_\mu a V_{\nu}^*) = \de_{\mu, \nu} e^{-|\mu|\be} \psi(a), 
\]
for all $a \in A$ and $\mu, \nu \in \bF_+^d$.
\end{inparaenum}
\end{proposition}

\begin{proof}
Let $\psi$ be a $(\si,\be)$-KMS state of $\T(A,\al)$.
Let the unit $1 \in A$, which is also a unit for $\T(A,\al)$.
Then we obtain
\begin{align*}
1 = \psi(1)
& \geq
\sum_{k=1}^d \psi(V_i V_i^*)
 =
\sum_{k=1}^d \psi(V_i^* \si_{i \be}(V_i))
 =
e^{-\be} \sum_{i=1}^d \psi(1)
=
e^{-\be} d,
\end{align*}
where we used that $1 - \sum_{i=1}^d V_iV_i^* \geq 0$; thus $\be \geq \log d$.

For item (ii) fix $\be \geq \log d$ such that $\be > 0$ and suppose first that $\psi$ is a $(\si,\be)$-KMS state.
Since $\si_t|_A = \id_A$ we obtain that $\psi|_A$ is tracial.
By applying the KMS condition twice we have that
\[
\psi(V_\mu a) = \psi(a \si_{i\be}(V_\mu)) = e^{-|\mu|\be} \psi(V_\mu a),
\]
therefore $\psi(V_\mu a) = 0$ for all $a \in A$ and $\mt \neq \mu \in \bF_+^d$.
By taking adjoints we have that $\psi(a V_\mu^*) = 0$ for all $a \in A$ and $\mt \neq \mu \in \bF_+^d$.
Similarly we have that $\psi(a V_\mu) = 0$ for all $a \in A$ and $\mt \neq \mu \in \bF_+^d$.
Therefore we have that
\begin{align*}
\psi(V_\mu a V_\nu^*)
=
\psi(a V_\nu^* \si_{i\be}(V_\mu))
=
e^{-|\mu|\be}\psi(a V_\nu^* V_\mu)
=
\de_{\mu, \nu} e^{-|\mu|\be} \psi(a)
\end{align*}
for all $a \in A$ and $\mu,\nu \in \bF_+^d$.
Thus a $(\si,\be)$-KMS state must have the form of the statement.

Conversely suppose that $\psi$ is as in the statement for some $\be \geq \log d$. We will verify the KMS condition. To this end let $f = V_\mu a V_\nu^*$ and $g = V_\ka b V_\la^*$. A direct computation shows that
\begin{align*}
fg
& =
\begin{cases}
V_\mu a V_{\nu'}^* b V_\la^* & \text{ when } \nu = \ka \nu', \\
V_\mu a V_{\ka'} b V_\la^* & \text{ when } \ka = \nu \ka', \\
0 & \text{ otherwise,}
\end{cases} \\
& =
\begin{cases}
V_\mu a \al_{\ol{\nu'}} (b)V_{\la\nu'}^*& \text{ when } \nu = \ka \nu', \\
V_{\mu \ka'} \al_{\ol{\ka'}}(a) b V_\la^* & \text{ when } \ka = \nu \ka', \\
0 & \text{ otherwise.}
\end{cases}
\end{align*}
In a similar way we have that
\begin{align*}
gf
& =
\begin{cases}
V_\ka b \al_{\ol{\la'}} (a) V_{\nu\la'}^*& \text{ when } \la = \mu \la', \\
V_{\ka \nu'} \al_{\ol{\mu'}}(b) a V_\nu^* & \text{ when } \mu = \la \mu', \\
0 & \text{ otherwise.}
\end{cases}
\end{align*}
Referring to the comments preceding the statement, we aim to show that $\psi(fg) = e^{-(|\mu|-|\nu|)\be} \psi(gf)$.

First we deal with the case where $\nu = \ka \nu'$ and $\la = \mu \la'$.
By assumption we have that
\begin{align*}
\psi(fg) = \de_{\mu, \la \nu'} e^{-|\mu|\be} \psi(a \al_{\ol{\nu'}}(b)),
\end{align*}
and
\begin{align*}
\psi(gf) = \de_{\ka, \nu \la'} e^{-|\ka| \be} \psi(b \al_{\ol{\la'}}(a)) = \de_{\ka, \nu \la'} e^{-|\ka| \be} \psi(\al_{\ol{\la'}}(a) b),
\end{align*}
since $\psi|_A$ is tracial.

\smallskip

\noindent {\bf Claim.} Under the assumption that $\nu = \ka \nu'$ and $\la = \mu \la'$, we have that $\mu = \la \nu'$ if and only if $\ka = \nu\la'$.
Each of them implies that $\nu' = \la' = \mt$, $\mu = \la$ and $\ka = \nu$.

\smallskip

\noindent {\bf Proof of the Claim.} If $\mu = \la \nu'$ then we obtain that $\la = \mu \la' = \la \nu' \la'$.
Therefore $\nu' = \la' = \emptyset$ and as a consequence we obtain that $\nu = \ka \nu' = \ka$, thus $\ka = \nu \la'$.
The converse follows by symmetry which completes the proof of the claim.

\smallskip

\noindent Therefore if $\mu = \la \nu'$ or $\ka = \nu\la'$ then we have that $\psi(fg) = e^{-|\mu|\be} \psi(a b)$ and
\begin{align*}
e^{-(|\mu|-|\nu|)\be} \psi(gf)
& =
e^{-(|\mu| - |\nu|) \be} e^{-|\ka| \be} \psi(a b)
=
e^{-|\mu|\be} \psi(a b),
\end{align*}
since $|\mu| - |\nu| + |\ka| = |\mu|$.
Otherwise we get that $\psi(fg) = 0 = \psi(gf)$ which satisfies the KMS condition trivially.

\medskip

Secondly we deal with the case where $\nu = \ka \nu'$ and $\mu = \la \mu'$.
By assumption we have that
\begin{align*}
\psi(fg) = \de_{\mu, \la \nu'} e^{-|\mu|\be} \psi(a \al_{\ol{\nu'}}(b)),
\end{align*}
and
\begin{align*}
\psi(gf) = \de_{\nu, \ka\nu'} e^{-|\nu| \be} \psi(\al_{\ol{\mu'}}(b) a) = \de_{\nu, \ka\nu'} e^{-|\nu| \be} \psi(a \al_{\ol{\mu'}}(b)),
\end{align*}
since $\psi|_A$ is tracial.

\smallskip

\noindent {\bf Claim.} Under the assumption that $\nu = \ka \nu'$ and $\mu = \la \mu'$, we have that $\mu = \la \nu'$ if and only if $\nu = \ka\mu'$. Each of them implies that $\nu' = \mu'$.

\smallskip

\noindent {\bf Proof of the Claim.} If $\mu = \la \nu'$ then $\la \mu' = \la \nu'$, which implies that $\mu' = \nu'$; thus $\nu = \ka \nu' = \ka \mu'$.
The converse follows by symmetry and the proof of the claim is complete.

\smallskip

\noindent Therefore if $\mu = \la \nu'$ or $\nu = \ka\mu'$ we obtain that $\psi(fg) = e^{-|\mu| \be} \psi(a \al_{\ol{\nu'}}(b))$ and
\begin{align*}
e^{-(|\mu|-|\nu|)\be} \psi(gf)
=
e^{-(|\mu| - |\nu|)\be} e^{-|\nu| \be} \psi(a \al_{\ol{\mu'}}(b))
=
e^{-|\mu| \be} \psi(a \al_{\ol{\nu'}}(b)).
\end{align*}
Otherwise we have that $\psi(fg) = 0 = \psi(gf)$ and the KMS condition is satisfied trivially.

\medskip

The case where $\ka = \nu \ka'$ reduces to the previous computation by substituting the roles of $f$ and $g$ by $g^*$ and $f^*$ respectively, for which we have just showed that the KMS equation holds.
Finally in the above computations we also established that if $\nu \neq \ka \nu'$ or $\ka \neq \nu \ka'$ then we cannot have that $\la = \mu \la'$ and $\ka = \nu \la'$, or that $\mu= \la \mu'$ and $\ka \mu' = \nu$.
In these cases the KMS condition is satisfied trivially.
\end{proof}

We mention here a subtle point in item (ii) of Proposition \ref{P: constr KMS}.
The only case where a KMS state does not fall in this characterization is when $d = 1$ and $\be = \log d = 0$.
However in this case a KMS state is simply a tracial state and we will exhibit a different path.
On the other hand when $d > 1$ then the arguments of item (ii) of Proposition \ref{P: constr KMS} still hold for the critical temperature $\be =\log d > 0$.
We will leave these cases for later.
At this point we continue with our examination for the case $\be > \log d$, where we may apply item (ii) of Proposition \ref{P: constr KMS} for any finite multiplicity $d$.

\begin{proposition}\label{P: constr KMS}
Let $(A,\al)$ be a unital C*-dynamical system of multiplicity $d$ and let $\be > \log d$.
Then for every tracial state $\tau$ of $A$ there exists a $(\si, \be)$-KMS state $\psi_\tau$ of $\T(A,\al)$ such that
\begin{align*}
\psi_\tau(V_\mu a V_\nu^*)
=
\de_{\mu, \nu} \cdot (1 - e^{-\be}d) \cdot
\sum_{m=0}^\infty e^{-(m + |\mu|)\be} \sum_{|w| = m} \tau\al_{\ol{w}}(a),
\end{align*}
for all $a\in A$ and $\mu, \nu \in \bF_+^d$.
\end{proposition}

\begin{proof}
Fix a tracial state $\tau$ of $A$ and let $(H_\tau, \pi_\tau, \xi_\tau)$ be the GNS construction associated with $\tau$.
Then the pair $(\wt{\pi_\tau}, V_\tau)$ on $H_\tau \otimes \ell^2(\bF_+^d)$ defines a representation $V_\tau \times \wt{\pi_\tau}$ of $\T(A,\al)$.
For every word $\ka \in \bF_+^d$ let the vector states
\[
\psi_\ka(f) = \sca{(V_\tau \times \wt{\pi_\tau})(f) \xi_\tau \otimes e_\ka, \xi_\tau \otimes e_\ka}, \foral f \in \T(A,\al).
\]
We define the functional $\psi_\tau \colon \T(A,\al) \to \bC$ by
\[
\psi_\tau(f):= (1 - e^{-\be}d) \sum_{m=0}^\infty e^{-m\be} \sum_{|\ka|=m} \psi_\ka(f).
\]
Indeed, since $e^{-\be} d< 1$ we have that $\psi_\tau$ is well defined as a norm limit of vector states.
Moreover $\psi_\tau$ is positive since every finite sum is so as a sum of positive vector states.
If $f = 1_A$ then we get $\psi_\ka(1_A) = 1$ hence
\begin{align*}
\psi_\tau(1_A)
& =
(1 - e^{-\be} d) \sum_{m=0}^\infty e^{-m\be} \sum_{|\ka|=m} \psi_\ka(1_A) \\
& =
(1 - e^{-\be} d) \sum_{m=0}^\infty e^{-m\be} d^m \\
& =
(1 - e^{-\be} d) \sum_{m=0}^\infty (e^{-\be} d)^m = 1,
\end{align*}
since by assumption $e^{-\be} d< 1$.
Thus $\psi_\tau$ is a state on $\T(A,\al)$.

Next we show that $\psi_\tau$ is as in the statement.
For $f = V_\mu a V_\nu^*$ with $\mu \neq \nu$ we directly verify that $\psi_\ka(f) = 0$. 
For $f = V_\mu a V_\mu^*$ we have that
\begin{align*}
\psi_\ka(f)
& =
\sca{(V_\tau)_\mu \, \wt{\pi_\tau}(a) \, (V_\tau)_\mu^* \, \xi_\tau \otimes e_{\ka}, \xi_\tau \otimes e_{\ka}} \\
& =
\sca{\widetilde{\pi_\tau}(a) \, (V_\tau)_\mu^* \, \xi_\tau \otimes e_\ka, (V_\tau)_\mu^* \, \xi_\tau \otimes e_\ka} \\
& =
\begin{cases}
\sca{\pi_\tau\al_{\ol{\ka'}}(a) \xi_\tau, \xi_\tau} & \text{ when } \ka = \mu \ka', \\
0 & \text{ otherwise},
\end{cases} \\
& =
\begin{cases}
\tau\al_{\ol{\ka'}}(a) & \text{ when } \ka = \mu \ka', \\
0 & \text{ otherwise}.
\end{cases}
\end{align*}
Therefore we obtain that
\begin{align*}
\psi_\tau(f)
& =
(1 - e^{-\be} d) \sum_{m=0}^\infty e^{-m\be} \sum_{|\ka|=m} \psi_\ka(f) \\
& =
(1 - e^{-\be} d) \sum_{m=|\mu|}^\infty e^{-m\be} \sum_{|\ka|=m, \, \ka = \mu \ka', \ka' \in \bF_+^d} \tau\al_{\ol{\ka'}}(a) \\
& =
(1 - e^{-\be} d) \sum_{m=0}^\infty e^{-(m + |\mu|)\be} \sum_{|w|=m} \tau\al_{\ol{w}}(a).
\end{align*}
Finally we show that $\psi_\tau$ is a $(\si,\be)$-KMS state.
A direct computation shows that
\[
e^{-|\mu|\be} \psi_\tau(a) = e^{-|\mu| \be} (1 - e^{-\be} d) \sum_{m=0}^\infty e^{-m\be} \sum_{|w|=m} \tau\al_{\ol{w}}(a) = \psi_\tau(V_\mu a V_\mu^*),
\]
and the proof is then completed by Proposition \ref{P: cond KMS}.
\end{proof}

In the following theorem we show that the $(\si,\be)$-KMS states are exactly of this form when $\be > \log d$.

\begin{theorem}\label{T: KMS TP}
Let $(A,\al)$ be a unital C*-dynamical system of multiplicity $d$ and $\be > \log d$.
Then there is an affine weak*-homeomorphism $\tau \mapsto \psi_\tau$ from the simplex of the tracial states on $A$ onto the simplex of the $(\si,\be)$-KMS states on $\T(A,\al)$ with
\begin{align*}
\psi_\tau(V_\mu a V_\nu^*)
=
\de_{\mu, \nu} \cdot (1 - e^{-\be}d) \cdot
\sum_{m=0}^\infty e^{-(m + |\mu|)\be} \sum_{|w| = m} \tau\al_{\ol{w}}(a),
\end{align*}
for all $a\in A$ and $\mu, \nu \in \bF_+^d$.
\end{theorem}

\begin{proof}
The fact that $\tau \mapsto \psi_\tau$ is an affine weak*-continuous mapping follows by the standard arguments of \cite[Proof of Theorem 6.1]{LRRW13}.

First we show that the mapping is onto.
To this end, given a $(\si,\be)$-KMS state $\vpi$ of $\T(A,\al)$ we will construct a tracial state $\tau$ of $A$ such that $\vpi = \psi_\tau$.
By Proposition \ref{P: cond KMS} it suffices to show that $\vpi(a) = \psi_\tau(a)$ for all $a\in A$.
In what follows the key is to isolate a projection in $\T(A,\al)$ that commutes with $A$.
To this end we will use the projection
\[
P := I - \sum_{i=1}^d V_i V_i^* \in \T(A,\al).
\]
Indeed recall that
\[
a V_i V_i^* = V_i \al_i(a) V_i^* = V_i V_i^* a, 
\]
hence $Pa = aP$ for all $a \in A$.
Moreover
\begin{align*}
\vpi(P)
& =
1 - \sum_{i=1}^d \vpi(V_iV_i^*)
 =
1 - \sum_{i=1}^d e^{-\be} \vpi(V_i^*V_i) = 1 - e^{-\be} d.
\end{align*}
Notice here that the quantity $\vpi(P)$ is constant for all $(\si,\be)$-KMS states $\vpi$.
We aim to show that the function
\[
\vpi_P \colon A \to \bC: a \mapsto \frac{\vpi(PaP)}{\vpi(P)}
\]
gives the appropriate tracial state on $A$.
Since $\si_{i\be}(P) = P$ and $\si_{i\be}(a) = a$ we have that
\[
\vpi(PabP) = \vpi(b P P a) = \vpi(PbaP).
\]
Thus $\vpi_P$ is indeed a tracial state on $A$.
Let the elements
\[
p_m := \sum_{k=0}^m \sum_{|w| = k} V_w P V_w^* \in \T(A,\al).
\]
By definition we have that $(1 - \sum_{i=1}^d V_i V_i^*) V_w = 0$, hence $P V_u^* V_w P =0$ when $w \neq u$.
Therefore each $p_m$ is a projection.
Furthermore we obtain
\begin{align*}
\vpi(p_m)
& =
\sum_{k=0}^m \sum_{|w| = k} \vpi(V_w P V_w^*) \\
& =
\sum_{k=0}^m \sum_{|w|=k} e^{-|w|\be} \vpi(P) \\
& =
(1 - e^{-\be} d) \sum_{k=0}^m \sum_{|w|=k} e^{-|w|\be} \\
& =
(1 - e^{-\be} d) \sum_{k=0}^m (e^{-\be} d)^k \\
& =
1 - (e^{-\be} d)^{m+1}.
\end{align*}
Thus $\lim_m \vpi(p_m) =1$ and by \cite[Lemma 7.3]{LRR11} we get that $\lim_m \vpi(p_m f p_m) = \vpi(f)$ for all $f \in \T(A,\al)$. In particular for $a \in A$ we get that
\begin{align*}
\vpi(a)
& =
\lim_m \vpi(p_m a p_m) \\
& =
\lim_m \sum_{k=0}^m \sum_{l=0}^m \sum_{|w|=k} \sum_{|u|=k} \vpi(V_w P V_w^* a V_u P V_u^*) \\
& =
\lim_m \sum_{k=0}^m \sum_{l=0}^m \sum_{|w|=k} \sum_{|u|=k} e^{-|w|\be} \vpi(PV_w^* a V_u P V_u^* V_w P) \\
& =
\lim_m \sum_{k=0}^m \sum_{|w|=k} e^{-|w|\be} \vpi(P V_w^* a V_w P) \\
& =
\lim_m \sum_{k=0}^m \sum_{|w|=k} e^{-|w|\be} \vpi(P \al_{\ol{w}}(a) P) \\
& =
\sum_{k=0}^\infty \sum_{|w|=k} e^{-|w|\be} \vpi_P\al_{\ol{w}}(a),
\end{align*}
which shows that $\vpi = \psi_\tau$ for $\tau = \vpi_P$.

Finally we show that the mapping is one-to-one. 
To this end it suffices to show that if $\psi_\tau$ is the state associated with a tracial state $\tau$ of $A$ as in Proposition \ref{P: constr KMS} then $(\psi_\tau)_P(a) = \tau(a)$.
Indeed if this is true then $\psi_{\tau} = \psi_{\tau'}$ will imply that $\tau = \tau'$.
Since $\psi_\tau(P) = 1 - e^{-\be} d$ for any $\tau$, it suffices to show that
\[
\psi_\tau(PaP) = (1 - e^{-\be} d) \tau(a), \foral a \in A.
\]
A direct computation shows that
\begin{align*}
\psi_\tau(PaP)
& =
\psi_\tau(a) - \sum_{i=1}^d \big( \psi_\tau(V_iV_i^*a) + \psi_\tau(aV_iV_i^*) \big) + \sum_{i,j=1}^d \psi_\tau(V_iV_i^* a V_jV_j^*) \\
& =
\psi_\tau(a) - \sum_{i=1}^d \big( e^{-\be} \psi_\tau(V_i^*aV_i) - e^{-\be}\psi_\tau(V_i^*aV_i) + e^{-\be} \psi_\tau(V_i^* a V_i) \big)\\
& =
\psi_\tau(a) - \sum_{i=1}^d e^{-\be} \psi_\tau \al_i(a),
\end{align*}
where we have used that
\[
\psi_\tau(V_i \cdot V_i^*a) = e^{-\be} \psi_\tau(V_i^*a V_i), \,
\psi_\tau(aV_i \cdot V_i^*) = e^{-\be} \psi_\tau(V_i^*a V_i),
\]
and that 
\[
\psi_\tau(V_i \cdot V_i^* a V_jV_j^*) = \de_{i,j} e^{-\be} \psi_\tau(V_i^*a V_i).
\]
By the definition of $\psi_\tau$ we then obtain
\begin{align*}
\psi_\tau(PaP)
& =
(1 - e^{-\be} d) \sum_{m=0}^\infty \Big( \sum_{|w|=m} \tau\al_{\ol{w}}(a) - \sum_{i=1}^d \tau\al_{\ol{w}}\al_i(a) \Big) \\
& =
(1 - e^{-\be} d) \sum_{m=0}^\infty \Big( \sum_{|w|=m} \tau\al_{\ol{w}}(a) - \sum_{i=1}^d \tau\al_{\ol{iw}}(a) \Big)\\
& =
(1 - e^{-\be} d) \sum_{m=0}^\infty \Big( \sum_{|w|=m} \tau\al_{\ol{w}}(a) - \sum_{|w| = m+1} \tau\al_{\ol{w}}(a) \Big)\\
& =
(1 - e^{-\be} d) \tau(a),
\end{align*}
and the proof is complete.
\end{proof}


We continue with the examination of the ground states.

\begin{proposition}\label{P: ground TP}
Let $(A,\al)$ be a unital C*-dynamical system of multiplicity $d$.
Then the mapping $\tau \mapsto \psi_\tau$ with
\[
\psi_\tau(V_\mu a V_\nu^*) =
\begin{cases}
\tau(a) &\text{ for } \mu= \emptyset = \nu,\\
0 &\text{ otherwise},
\end{cases}
\]
is an affine weak*-homeomorphism from the state space $S(A)$ onto the space of the ground states on $\T(A,\al)$.
\end{proposition}

\begin{proof}
First we show that if $\psi$ is a ground state then it is of the aforementioned form.
Suppose that $\nu \neq \emptyset$; by assumption the map
\[
r + i t \mapsto \psi(V_\mu \si_{r + it}(a V_\nu^*)) = e^{-i|\nu|r} e^{|\nu|t} \psi(V_\mu a V_\nu^*),
\]
must be bounded when $t> 0$, for all $a \in A$.
Therefore $\psi(V_{\mu} a V_{\nu}^*) =0$ in this case.
When $\nu = \emptyset$ but $\mu \neq \emptyset$, then the function
\[
r + it \mapsto \psi(a^* \si_{r + i t}(V_\mu^*)) = e^{-i|\mu|r} e^{|\mu|t} \psi(a^* V_\mu^*)
\]
must be bounded for $t > 0$. This implies that $\psi(a^* V_\mu^*) = 0$.
Taking adjoints yields $\psi(V_\mu a) =0$.

Conversely let $\psi$ be a state on $\T(A,\al)$ that satisfies the condition of the statement. Then for $f = V_\mu a V_\nu^*$ and $g = V_\ka b V_\la^*$ we compute
\begin{align*}
|\psi(f \si_{r + it}(g))|^2
& =
|e^{i(|\ka| - |\la|)(r + it)} \psi(fg)|^2 \\
& =
e^{-(|\ka| - |\la|)2t} |\psi(fg)|^2 \\
& \leq
e^{-(|\ka| - |\la|)2t} \psi(f^*f) \psi(g^*g) \\
& \leq
e^{-(|\ka| - |\la|)2t} \psi(V_\nu a^* a V_\nu^*) \psi(V_\la b^*bV_\la^*).
\end{align*}
When $\nu \neq \emptyset$ or $\la \neq \emptyset$ then the above expression is $0$.
When $\nu = \la = \emptyset$ then
\begin{align*}
|\psi(f \si_{r + it}(g))|
& =
e^{-|\ka| t} |\psi(V_\mu a V_\ka b)|
 =
e^{-|\ka|t} |\psi(V_{\mu\ka} \al_{\ol{\ka}}(a) b)|,
\end{align*}
which is zero when $\mu \neq \emptyset$ or $\la \neq \emptyset$.
Finally when $\mu = \nu = \ka = \la = \emptyset$ then $\psi(f \si_{r + it}(g)) = \psi(ab)$, which is bounded for all $a,b \in A$.

To end the proof it suffices to show that every state on $A$ gives rise to a ground state on $\T(A,\al)$.
Fix $\tau \in S(A)$ and let $(H_\tau, \pi_\tau, \xi_\tau)$ be the associated GNS representation.
Then for the Fock representation $(\wt{\pi_\tau}, V_{\tau})$ we define the state
\[
\psi(f) = \sca{f \xi_\tau \otimes e_{\mt}, \xi_\tau \otimes e_{\mt}}, \foral f \in \T(A,\al).
\]
It is readily verified that $\psi(a) = \tau(a)$ for all $a\in A$.
For $f = V_\mu a V_\nu^*$ we compute
\begin{align*}
\psi(V_\mu a V_\nu^*)
& =
\sca{(V_\tau)_\mu \, \wt{\pi_\tau}(a) \, (V_\tau)_\nu^* \, \xi_\tau \otimes e_{\mt}, \xi_\tau \otimes e_{\mt}} \\
& =
\sca{\wt{\pi_\tau}(a) \, (V_\tau)_\nu^* \, \xi_\tau \otimes e_{\mt}, (V_\tau)_\mu^* \,  \xi_\tau \otimes e_{\mt}} \\
& =
\de_{\mu, \emptyset} \de_{\nu, \emptyset} \sca{\pi_\tau(a) \xi_\tau, \xi_\tau} \\
& =
\begin{cases}
\tau(a) &\text{ when } \mu = \emptyset = \nu,\\
0 &\text{ otherwise},
\end{cases} \\
& =
\begin{cases}
\psi(a) &\text{ when } \mu = \emptyset = \nu,\\
0 &\text{ otherwise},
\end{cases}
\end{align*}
and the proof is complete.
\end{proof}

We continue with the analysis of the KMS${}_\infty$ states on $\T(A,\al)$.

\begin{proposition}\label{P: infty TP}
Let $(A,\al)$ be a unital C*-dynamical system of multiplicity $d$.
Then the mapping $\tau \mapsto \psi_\tau$ with
\[
\psi_\tau(V_\mu a V_\nu^*)
=
\begin{cases}
\tau(a) &\text{ for } \mu = \emptyset = \nu,\\
0 &\text{ otherwise},
\end{cases}
\]
is an affine weak*-homeomorphism from the tracial state space $T(A)$ onto the space of the KMS${}_\infty$ states on $\T(A,\al)$.
\end{proposition}

\begin{proof}
First we show that if $\psi$ is a KMS${}_\infty$ state then it is of the aforementioned form.
Let $\psi_{\be}$ be $(\si,\be)$-KMS states on $\T(A,\al)$ for $\be > \log d$ that converge in the w*-topology to $\psi$.
By Proposition \ref{P: cond KMS} we obtain that every $\psi_\be|_A$ is a tracial state on $A$ and that
\[
\psi_{\be}(V_\mu a  V_\nu^*) = \de_{\mu, \nu} e^{-|\mu|\be} \psi_\be(a)
\]
which tends to zero when $\be \longrightarrow \infty$.
Consequently $\psi$ is as in the statement.

Conversely, let $\psi_\tau$ be as in the statement with respect to a tracial state $\tau$ of $A$.
Let $\psi_{\tau,\be}$ be as defined in Proposition \ref{P: constr KMS} for $\be > \log d$, i.e.
\begin{align*}
\psi_{\tau, \be}(V_{\mu} a V_{\nu}^*)
=
\de_{\mu, \nu} \cdot (1 - e^{-\be} d) \cdot
\sum_{m=0}^\infty e^{-(m + |\mu|)\be} \sum_{|w| = m} \tau\al_{\ol{w}}(a).
\end{align*}
By the w*-compactness we may choose a sequence of such states that converges to a state, say $\psi$.
By definition $\psi$ is then a KMS$_\infty$ state, and we aim to show that $\psi_\tau = \psi$. When $\mu \neq \emptyset$ or $\nu \neq \emptyset$ then we get that
\[
\psi(V_\mu a V_\nu^*) = \lim_{\be \rightarrow \infty} \psi_{\tau,\be}(V_\mu a V_\nu^*) = 0 = \psi_\tau(V_\mu a V_\nu^*),
\]
as in the preceding paragraph.
When $\mu = \nu = \emptyset$ we obtain
\begin{align*}
\psi_{\tau,\be}(V_\mu a V_\nu^*)
& =
(1 - e^{-\be} d) \cdot
\sum_{m=0}^\infty e^{-m\be} \sum_{|w| = m} \tau\al_{\ol{w}}(a) \\
& =
(1- e^{-\be} d) \cdot
\Big( \tau(a) + \sum_{m=1}^\infty e^{-m\be} \sum_{|w| = m} \tau\al_{\ol{w}}(a) \Big).
\end{align*}
However we have that
\begin{align*}
| \sum_{m=1}^\infty e^{-m\be} \sum_{|w| = m} \tau\al_{\ol{w}}(a) |
& \leq
\nor{a} \cdot \sum_{m=1}^\infty e^{-m\be} d^m \\
& =
\nor{a}(-1 + (1- e^{-\be} d)^{-1}).
\end{align*}
Taking $\be \longrightarrow \infty$ yields that the quantity $\sum_{m=1}^\infty e^{-m\be} \sum_{|w| = m} \tau\al_{\ol{w}}(a)$ tends to zero.
Since $\lim_{\be \rightarrow \infty} (1- e^{-\be} d) = 1$, we then obtain that $\lim_{\be \rightarrow \infty} \psi_{\tau,\be}(a) = \tau(a)$.
Therefore we get
\[
\psi(V_\mu a V_\nu^*) = \psi(a) = \lim_{\be \rightarrow \infty} \psi_{\tau,\be}(a) = \tau(a) = \psi_\tau(a) = \psi_\tau(V_\mu a V_\nu^*)
\]
which completes the proof that $\psi = \psi_\tau$ on $\T(A,\al)$.
The last part follows in the same way as in the proof of Proposition \ref{P: ground TP}.
\end{proof}

\subsection{The Cuntz-Pimsner algebra}

The gauge action on $\O(A,\al)$ induces an action of $\bR$ on $\O(A,\al)$ as in the case of $\T(A,\al)$.
We will denote it by the same symbol $\si$.
(The reason being that) the gauge actions on $\T(A,\al)$ and $\O(A,\al)$ intertwine the quotient map $q \colon \T(A,\al) \to \O(A,\al)$ on the ideal generated by
\[
a \cdot (I - \sum_{i=1}^d V_iV_i^*), \foral a \in J_{(A,\al)}.
\]
Recall here that $J_{(A,\al)} = (\bigcap_{i=1}^d \ker \al_i )^\perp$.
Therefore the $(\si,\be)$-KMS states on $\O(A,\al)$ define $(\si,\be)$-KMS states on $\T(A,\al)$.
In Theorem \ref{T: KMS CP} we show that there is actually a bijection.
As we observed in the introduction we cannot apply directly the analysis of Laca and Neshveyev \cite{LacNes04}.

\begin{theorem}\label{T: KMS CP}
Let $(A,\al)$ be a unital C*-dynamical system of multiplicity $d$ and $\be > \log d$.
Then there is an affine weak*-homeomorphism $\tau \mapsto \vpi_\tau$ from the simplex of the tracial states on $A$ that vanish on $J_{(A,\al)} = (\bigcap_{i=1}^d \ker \al_i )^\perp$ onto the simplex of the $(\si,\be)$-KMS states on $\O(A,\al)$ such that
\begin{align*}
\vpi_\tau(S_\mu a S_\nu^*)
=
\de_{\mu, \nu} \cdot (1 - e^{-\be}d) \cdot
\sum_{m=0}^\infty e^{-(m + |\mu|)\be} \sum_{|w| = m} \tau\al_{\ol{w}}(a),
\end{align*}
for all $a\in A$ and $\mu, \nu \in \bF_+^d$.
\end{theorem}

\begin{proof}
Let $\tau$ be a tracial state on $A$ and let $\psi_\tau$ be a state of $\T(A,\al)$ as in Proposition \ref{P: constr KMS}.
As in Theorem \ref{T: KMS TP} let the projection $P = I - \sum_{i=1}^d V_iV_i^*$ and recall that $a P =P a$ for all $a \in A$.
By the last computation of Theorem \ref{T: KMS TP} we have that
\begin{align*}
\psi_\tau(a P) = \psi_\tau(P a P) = (1 - e^{-\be} d) \tau(a)
\end{align*}
for all $a \in A$.
As a consequence, if $\tau$ vanishes on $J_{(A,\al)}$ then $\psi_\tau$ vanishes on $\ker q$ and a $(\si,\be)$-KMS state on $\O(A,\al)$ is defined by $\vpi_\tau q = \psi_\tau$.
Conversely if $\vpi_\tau$ is as in the statement then let $\psi_\tau = \vpi_\tau q$ and the above equation shows that $\tau$ vanishes on $J_{(A,\al)}$.
\end{proof}

\begin{proposition}\label{P: inj CP}
Let $(A,\al)$ be an injective unital C*-dynamical system of multiplicity $d$.
Then $\O(A,\al)$ does not attain $(\si,\be)$-KMS states for $\be \neq \log d$.
\end{proposition}

\begin{proof}
If $(A,\al)$ is injective then $J_{(A,\al)} = A$, hence $\sum_{i=1}^d S_i S_i^* = 1$ in $\O(A,\al)$.
Proceed as in the proof of Proposition \ref{P: cond KMS} to obtain an estimation for $\be$ by using the $S_i$ in the place of $V_i$.
However now we obtain equality which shows that $\be = \log d$.
\end{proof}


We conclude with the analogues for the ground states and the KMS${}_\infty$ states on $\O(A,\al)$.
The proofs are left to the reader.
We remark that in the case of the KMS${}_\infty$ states one has to make use of Theorem \ref{T: KMS CP}.

\begin{corollary}\label{C: ground/infty CP}
Let $(A,\al)$ be a unital C*-dynamical system.
Then the mapping $\tau \mapsto \vpi_\tau$ with
\[
\vpi_\tau(S_\mu a S_\nu^*) =
\begin{cases}
\tau(a) & \text{ for } \mu=\mt=\nu,\\
0 & \text{ otherwise},
\end{cases}
\]
is an affine weak*-homeomorphism from the simplex of the states $S(A)$ (resp. the tracial states $T(A)$) that vanish on $J_{(A,\al)} = (\bigcap_{i=1}^d \ker \al_i )^\perp$ onto the simplex of the ground states (resp. KMS${}_\infty$ states) on $\O(A,\al)$.
\end{corollary}

\subsection{KMS states at $\be = \log d$}

We continue with the examination of the $(\si,\log d)$-KMS states.
First we show the connection between $(\si,\log d)$-KMS states on $\T(A,\al)$ with $(\si,\log d)$-KMS states on $\O(A,\al)$.

\begin{proposition}\label{P: boundary}
Let $(A,\al)$ be a unital C*-dynamical system of multiplicity $d$.
Then $\psi$ is a $(\si,\log d)$-KMS state on $\T(A,\al)$ if and only if it factors through a $(\si,\log d)$-KMS state on $\O(A,\al)$.
\end{proposition}

\begin{proof}
If $\psi$ is a $(\si,\log d)$-KMS state on $\T(A,\al)$ then
\[
1 = \psi(1) = \psi(V_i^* V_i) = e^{\log d} \psi(V_i V_i^*) = d \psi(V_i V_i^*),
\]
for all $i=1, \dots, d$.
As in Theorem \ref{T: KMS TP} let the projection $P = I - \sum_{i=1}^d V_iV_i^*$ which is an element of $\T(A,\al)$.
By the Cauchy-Schwartz inequality we then obtain
\begin{align*}
|\psi( a P )|^2
& \leq
\psi(aa^*) \cdot \psi (P)
\end{align*}
for all $a\in A$, since $P$ is a projection in $\T(A,\al)$. However we have that
\[
\psi(P) = \psi(1 - \sum_{i=1}^d V_iV_i^*) = 1 - \sum_{i=1}^d \psi(V_i V_i^*) = 0
\]
which shows that $\psi( a (1 - \sum_{i=1}^d V_iV_i^*)) = 0$ for all $a\in A$ and in particular for all $a \in J_{(A,\al)}$.
Hence $\psi$ vanishes on $\ker q$ and thus defines a state $\vpi$ on $\O(A,\al)$.
The converse follows by the fact that the actions of $\bR$ on the Pimsner algebras intertwine $q$, and the proof is complete.
\end{proof}

One advantage of the approach of Laca and Raeburn \cite{LacRae10} is that the analysis of the KMS states of $\T(A,\al)$ implies existence of $(\si,\log d)$-KMS states on $\O(A,\al)$.
In short, begin with a sequence of $\{\psi_{n}\}$ such that every $\psi_n$ is a $(\si,\be_n)$-KMS state of $\T(A,\al)$ with $\be_n \downarrow \log d$.
By passing to a subsequence and re-labeling we may assume that the $\psi_n$ converge to a state $\psi$ in the weak*-topology.
Then $\psi$ is a $(\si,\log d)$-KMS state for $\T(A,\al)$ \cite[Proposition 5.3.23]{BraRob97} and Proposition \ref{P: boundary} gives the required state on $\O(A,\al)$.

This scheme has been applied in several cases to give details about the form of $(\si,\log d)$-KMS states.
For example we mention \cite[Theorem 7.3]{LRRW13}, where a concrete context is available.

However this route seems difficult to be pursued at the generality we aim here.
Our analysis is derived at the level where we make a distinction between injective and non-injective C*-dynamical systems.
Below we present an alternative pathway for producing directly $(\si,\log d)$-KMS states.
For doing so we use the tail adding technique from \cite{KakKat11} exhibited in Subsection \ref{Ss: dyn sys}.
For the next proposition let $(C,\ga)$ be the injective dilation of $(A,\al)$ with the same multiplicity $d$.
Also let $(C^{(1)},\ga^{(1)})$ be the unitization of $(C,\ga)$.

\begin{proposition}
Let $(A,\al)$ be a unital C*-dynamical system and let $(C,\ga)$ be its injective dilation as in \cite{KakKat11}.
Then a $(\si,\log d)$-KMS state on $\O(C,\ga)$ or $\O(C^{(1)},\ga^{(1)})$ defines a $(\si,\log d)$-KMS state on $\O(A,\al)$ by restriction.
\end{proposition}

\begin{proof}
By construction, we have that $A \subseteq C$ and that $p \ga_i(a) p  = \al_i(a)$ for all $a\in A$ and $p=1 \in A$.
In addition $\O(A,\al)$ is the full corner of $\O(C,\ga)$ by the projection $p$.
Moreover $\O(C,\ga)$ (and consequently $\O(A,\al)$) is a C*-subalgebra of $\O(C^{(1)},\ga^{(1)})$.
All these inclusions are canonical in the sense that if $(\pi,\{S_i\}_{i=1}^d)$ defines a faithful representation of $\O(C^{(1)},\ga^{(1)})$, then $(\pi|_C,\{S_i\}_{i=1}^d)$ defines a faithful representation of $\O(C,\ga)$ and $(\pi|_A,\{S_ip\}_{i=1}^d)$ defines a faithful representation of $\O(A,\al)$.
As a consequence the gauge action, and thus the action $\si$, is compatible with the inclusions.
\end{proof}

Therefore we may produce $(\si,\log d)$-KMS states from the injective dilations.
We will thus restrict now our attention to injective systems $(A,\al)$.

We examine separately the cases $d = 1$ and $d >1$.
The main reason is because when $d=1$ then the KMS states at critical temperature amount to tracial states.
Motivated by \cite{MWM98} we obtain the following result for $d>1$.

\begin{proposition}\label{P: affine log d}
Let $(A,\al)$ be an injective unital C*-dynamical system of multiplicity $d>1$.
We write $\tau \in AVT(A,\al)$ for the $\tau \in T(A)$ such that
\[
\tau(a) = \frac{1}{d} \sum_{i=1}^d \tau \al_i(a), \foral a\in A,
\]
and we write $\vpi \in AVT(\O(A,\al)^\ga)$ for the $\vpi \in T(\O(A,\al)^\ga)$ such that
\[
\vpi(f) = \frac{1}{d} \sum_{i=1}^d \vpi(S_i^* f S_i), \foral f \in \O(A,\al)^\ga.
\]
Then there is an affine weak*-homeomorphism between the three simplices of $AVT(A,\al)$, of $AVT(\O(A,\al)^\ga)$, and that of the $(\si,\log d)$-KMS states on $\O(A,\al)$.
\end{proposition}

\begin{proof}
Recall that the fixed point algebra is $\O(A,\al)^\ga = \ol{\cup_n A_n}$ where $A_n = \spn\{S_\mu a S_\nu^* \mid |\mu|, |\nu| \leq n\}$.
On the other hand let the C*-subalgebras
\[
B_n = \spn\{ S_\mu a S_\nu^* \mid |\mu| = |\nu| = n\}
\]
of $\O(A,\al)^\ga$ so that $A_n = \cup_{m=0}^n B_m$.
Since the system is injective and unital we may use that $1 = \sum_{i=1}^d S_iS_i^*$ to obtain
\[
B_{m} \ni S_\mu a S_\nu^* = S_\mu a \cdot 1 \cdot S_\nu^* = \sum_{i=1}^d S_{\mu i} \al_i(a) S_{\nu i}^* \in B_{m+1}.
\]
Consequently $A_n = B_n$ for all $n \in \bZ_+$ and $\O(A,\al)^\ga$ is the inductive limit of the $B_n$ for $n \in \bZ_+$ via the connecting $*$-homomorphisms
\[
B_{n} \ni S_\mu a S_\nu^* \mapsto \sum_{i=1}^d S_{\mu i} \al_i(a) S_{\nu i}^* \in B_{n+1}.
\]
In what follows we will use this form of $\O(A,\al)^\ga$.

For $\tau \in AVT(A,\al)$ we define the extension $\vpi_n$ on $B_n$ by
\[
\vpi_n(S_\mu a S_\nu^*) := \frac{1}{d^n} \sum_{|w| = |\mu|} \tau(S_w^* S_\mu a S_\nu^* S_w) = \frac{1}{d^n} \de_{\mu,\nu} \tau(a).
\]
Every $\vpi_n$ is positive being the average on the diagonal of $B_n$.
The family $\{\vpi_n\}_n$ is compatible with the directed sequence on the $B_n$, because
\begin{align*}
\vpi_{n+1}\big(\sum_{i=1}^d S_{\mu i} \al_i(a) S_{\nu i}^*\big)
& =
\sum_{i=1}^d \frac{1}{d^{n+1}} \de_{\mu i, \nu i} \tau \al_i(a) \\
& =
\de_{\mu,\nu} \frac{1}{d^n} \sum_{i=1}^d \frac{1}{d} \tau \al_i(a)
=
\vpi_n(S_\mu a S_\nu^*).
\end{align*}
Therefore $\{\vpi_n\}_n$ defines a positive functional $\vpi_\tau \colon \O(A,\al)^\ga \to \bC$.
The unit $1 \in A$ is mapped to $f_{n} := \sum_{|\nu| = n} S_\nu S_\nu^*$ via the inclusion $A \hookrightarrow B_n$.
The computation
\begin{align*}
\vpi_n(f_n)
=
\frac{1}{d^n} \sum_{|w| = |\mu|} \sum_{|v| = n} \tau(S_w^* S_v S_v^* S_w)
=
\frac{1}{d^n} \sum_{|w| = |v| =n} \de_{w,v} \tau(1)
=
1
\end{align*}
shows that each $\vpi_n$ is a state and consequently $\vpi$ is a state.
In addition, for the elements $S_\mu a S_\nu^*, S_\ka b S_\la^* \in B_n$ we have
\[
\vpi_n(S_\mu a S_\nu^* \cdot S_\ka b S_\la^*)
=
\de_{\nu, \ka} \vpi_n(S_\mu a b S_\la^*)
=
\de_{\nu,\ka} \de_{\mu,\la} \frac{1}{d^n} \tau(a b)
\]
which gives a symmetrical formula (since $\tau$ is tracial), and implies that every $\vpi_n$ is tracial.
Thus $\vpi_\tau$ is a tracial state on $\O(A,\al)^\ga$.
Now let $f = S_\mu a S_\nu^* \in \O(A,\al)^\ga$ and compute
\begin{align*}
\frac{1}{d} \sum_{i=1}^d \vpi_\tau(S_i^* f S_i)
& =
\frac{1}{d} \sum_{i=1}^d \vpi_\tau(S_i^* S_\mu a S_\nu^*  S_i) \\
& =
\frac{1}{d} \sum_{i=1}^d \de_{\mu, i \mu'} \de_{\nu, i\nu'} \vpi_\tau(S_{\mu'} a S_{\nu'}^*) \\
& =
\frac{1}{d} \sum_{i=1}^d \de_{\mu, i \mu'} \de_{\nu, i\nu'} \de_{\mu,\nu} \frac{1}{d^{|\mu| - 1}} \tau(a) \\
& =
\frac{1}{d^{|\mu|}} \tau(a) = \vpi_\tau(f)
\end{align*}
which shows that $\vpi_\tau \in AVT(\O(A,\al)^\ga)$.
Due to the form of $\vpi_\tau$ we obtain that the mapping $\tau \mapsto \vpi_\tau$ is one-to-one.
In order to show injectivity fix $\vpi \in AVT(\O(A,\al)^\ga)$ and compute
\begin{align*}
\vpi(a)
=
\frac{1}{d} \sum_{i=1}^d \vpi(S_i^* a S_i)
=
\frac{1}{d} \sum_{i=1}^d \vpi(S_i^*S_i\al_i(a))
=
\frac{1}{d} \sum_{i=1}^d \vpi \al_i(a),
\end{align*}
hence $\tau = \vpi|_A \in AVT(A,\al)$.

For the second part fix $\vpi_\tau \in AVT(\O(A,\al)^\ga)$ for some $\tau \in AVT(A,\al)$ and define the state $\psi_\tau = \vpi_\tau E \colon \O(A,\al) \to \bC$ where $E$ is the conditional expectation onto the fixed point algebra.
Then $\psi|_A = \tau$ is tracial and
\[
\psi_\tau(S_\mu a S_\nu^*) = \de_{|\mu|,|\nu|} \vpi_\tau(S_\mu a S_\nu^*) = \de_{|\mu|, |\nu|} \de_{\mu, \nu} \frac{1}{d^{|\mu|}} \tau(a) = \de_{\mu,\nu} \frac{1}{d^{|\mu|}} \psi_\tau(a).
\]
By Proposition \ref{P: cond KMS} and Proposition \ref{P: boundary} we obtain that $\psi_\tau$ is a $(\si, \log d)$-KMS state on $\O(A,\al)$.
Conversely if $\psi$ is a $(\si,\log d)$-KMS state on $\O(A,\al)$ let $\tau = \psi|_A$.
Then $\tau$ is tracial and a similar computation as before yields $\tau \in AVT(A,\al)$.
By the form of $\psi$ we obtain that the mapping $\tau \mapsto \psi_\tau$ is an affine isomorphism.

Furthermore due to the formulas of $\vpi_\tau$ and $\psi_\tau$ we obtain that the affine isomorphisms are weak*-homeomorphisms.
\end{proof}

\begin{remark}\label{R: comm logd CP}
We use Proposition \ref{P: affine log d} and the ideas of \cite[Section 4]{MWM98} to show directly that $\O(A,\al)$ attains KMS states at the critical temperature $\log d$, when $(A,\al) \equiv (X,\si)$ is a classical system of multiplicity $d>1$ such that $X$ is a compact Hausdorff space with $\ol{\cup_{i=1}^d \si_i(X)} = X$.
Notice here that the assumption that $\ol{\cup_{i=1}^d \si_i(X)} = X$ is equivalent to $(A,\al)$ being injective.
In short, let the positive contractive map $\la \colon A \to A$ such that
\[
\la(a) = \frac{1}{d} \sum_{i=1}^d \al_i(a).
\]
The spectral radius of $\la$ equals $1 \in \bR$ because $(\la - \id_A)(1_A) = 0$ so that $1 \in \bC$ is in the spectrum of $\la$.
Let us denote by $\la^\# \colon A^\# \to A^\#$ the induced operator on the continuous linear functionals of $A$.
Since $\textup{sp}(\la) = \textup{sp}(\la^\#)$ we have that the spectral radius of $\la^\#$ is $1 \in \bR$.
Therefore for the resolvent $R^\#(t) = (t - \la^\#)^{-1}$ of $\la^\#$ there is a $\vpi_0 \in A^\#$ such that $\nor{R^\#(t) \vpi_0}$ is unbounded as $t \downarrow 1$.
By the Jordan decomposition we may assume that $\vpi_0$ is a state and let the states
\[
\vpi_n = \frac{R^\#(1 + \frac{1}{n})\vpi_0}{\nor{R^\#(1 + \frac{1}{n})\vpi_0}} \qfor n \geq 1.
\]
By weak*-compactness let $\vpi$ be an accumulation point of $(\vpi_n)$ for which we obtain that $\vpi = \la^\# \vpi = \vpi \la$.
Therefore $\vpi = \frac{1}{d} \sum_{i=1}^d \vpi \al_i$ and moreover $\vpi$ is a tracial state; consequently $\vpi \in AVT(A,\al)$.

There are cases where the $(\si,\log d)$-KMS state is unique.
Suppose that $A$ is commutative and finite dimensional, and let $\la \colon A \to A$ defined by $\la(a) = \frac{1}{d} \sum_{i=1}^d \al_i(a)$.
We say that $\la$ is \emph{irreducible} if $\la(J) \subseteq J$ holds only for the trivial ideals $J$ of $A$.
In this case the system $(A,\al)$ is injective and $\la$ corresponds to a non-negative irreducible matrix.
By the Perron-Frobenius Theorem then the eigenvalue $1 \in \bC$ is simple and thus $\vpi$ is unique.
\end{remark}

Finally for $d=1$ we have that the KMS states coincide with the tracial states by definition.
For the next proposition let $(\wh{C},\wh{\ga})$ be the automorphic dilation of $(A,\al)$ with the same multiplicity $d=1$ as exhibited in Subsection \ref{Ss: dyn sys}.

\begin{proposition}\label{P: tracial crossed}
Let $\al \colon \bZ_+ \to \End(A)$ be a unital C*-dynamical system and let $\wh{\ga} \colon \bZ_+ \to \Aut(\wh{C})$ be its automorphic dilation as in \cite{Kak11-1}.
For any tracial state $\tau$ of $\wh{C}$ there exists a tracial state $\psi$ of $\O(A,\al)$ such that
\begin{align*}
\psi (S_n a S_m^*) = \de_{n, m} \tau\wh{\ga}^{-n}(a) \foral a\in A \text{ and } n, m \in \bZ_+.
\end{align*}
\end{proposition}

\begin{proof}
Recall that $\O(A,\al)$ is a full corner of $\wh{C} \rtimes_{\wh{\ga}} \bZ$ \cite{Kak11-1}.
Hence a tracial state on the crossed product defines by restriction a tracial state on $\O(A,\al)$.
If $\tau$ is a tracial state on $\wh{C}$ then let $\psi : = \tau E \colon \wh{C} \rtimes_{\wh{\ga}} \bZ \to \bC$ where $E$ is the conditional expectation of the crossed product.
It is then readily verified that $\psi|_{\O(A,\al)}$ satisfies the condition of the statement.
\end{proof}

\begin{remark}
The difference between Proposition \ref{P: affine log d} and Proposition \ref{P: tracial crossed} is in having a parametrization for the states.
A review of the proof of Proposition \ref{P: cond KMS} highlights this phenomenon.
Even though the tracial states of Proposition \ref{P: tracial crossed} satisfy $\tau = \tau \ga$ (where $(C,\ga)$ is the injective dilation of $(A,\al)$) these are far from being the only choices.
For example for $d=1$ if $A = \bC$ and $\al = \id$ then $\O(A,\al) = \rC(\bT)$ and obviously there are plenty of other tracial states than the trivial one.
However for $d > 1$ and $\al_i = \id$ then $\O(A,\al) = \O_d$ has a unique $(\si, \log d)$-KMS state given as in Remark \ref{R: comm logd CP} \cite{MWM98}.
\end{remark}

\section{Applications}

Let $(A,\al) \equiv (X,\si)$ and $(C,\ga) \equiv (Y,\rho)$ be two classical systems on the compact Hausdorff spaces $X$ and $Y$.
There are several notions of equivalences for $(X,\si)$ and $(Y,\rho)$.

We say that $(X,\si)$ and $(Y,\rho)$ are \emph{unitarily equivalent} if there is a homeomorphism $\phi \colon X \to Y$ and a unitary $U \in M_{d_\si, d_\rho}(C(Y))$ that intertwines the diagonal actions
\[
\diag_{\phi \si}(f):= \diag(f \phi \si_i \mid i=1, \dots, d_\si)
\]
and
\[
\diag_{\rho \phi}(f) : = \diag(f \rho_i \phi \mid i=1, \dots, d_\rho).
\]
As a consequence the multiplicities $d_\si$ and $d_\rho$ must coincide.
Unitary equivalence of $(X,\si)$ and $(Y,\rho)$ is in fact unitary equivalence of the associated C*-correspondences.
Therefore unitarily equivalent systems have isomorphic \emph{tensor algebras} and Pimsner algebras (and so the Pimsner algebras attain the same KMS-theory).
Recall that the tensor algebra of $(A,\al)$ in the sense of Muhly and Solel \cite{MuhSol98} is the closed subalgebra of $\T(A,\al)$ generated by $\Bv_w a$ with $a\in A$ and $w \in \bF_+^d$.
The author and Katsoulis \cite[Theorem 5.2]{KakKat12} showed a converse: unitary equivalence is a complete invariant for isometric isomorphic tensor algebras of multivariable classical systems.
Such a result does not hold for the Cuntz-Pimsner algebras.
For $d=1$ unitary equivalence is just conjugacy of the systems and Hoare and Parry \cite{HoaPar66} have given a counterexample of a homeomorphism that is not conjugate to its inverse (however the crossed products are $*$-isomorphic).

Davidson and Katsoulis \cite[Definition 3.16]{DavKat11} introduced the notion of piecewise conjugacy.
Two classical systems $(X,\si)$ and $(Y,\rho)$ are called \emph{piecewise conjugate} if they have the same multiplicities and there is a homeomorphism $\phi \colon X \to Y$ such that for every $x \in X$ there is a neighborhood $\U$ of $x$ and a permutation $\pi$ on $d$ symbols so that
\[
\rho_i \phi|_\U = \phi \si_{\pi(i)}|_\U, \foral i=1, \dots, d.
\]
Piecewise conjugacy is weaker than unitary equivalence and let us include a short proof for completeness of the discussion.

\begin{proposition}\label{P: un eq pc}
If $(A,\al) \equiv (X,\si)$ and $(C,\ga) \equiv (Y,\rho)$ are unitarily equivalent then they are piecewise conjugate.
\end{proposition}

\begin{proof}
Since the systems are unitarily equivalent, they have the same multiplicities.
For simplicity let $d = d_\si = d_\rho$.
Suppose that there is a homeomorphism $\phi \colon X \to Y$ and a unitary $U = [u_{ij}] \in M_d(C(Y))$ such that
\[
\diag_{\phi \si}(f) U = U \diag_{\rho \phi}(f), \foral f \in C(Y).
\]
By substituting $\rho_i$ with $\phi^{-1} \rho_i \phi$ we may assume that $X=Y$ and $\phi = \id$.
Fix a point $x \in X$ and set $U^{(1)} = U$.
Since $U(x) = [u_{ij}(x)]$ is a unitary in $M_d(\bC)$ there is at least one $u_{1j}$ such that $u_{1j}(x) \neq 0$.
Choose a neighborhood $\U_1$ of $x$ such that $u_{1j}|_{\U_1} \neq 0$.
By using this element perform the first step of the Gaussian elimination as in \cite[Lemma 3.3]{KakKat12} to obtain a second invertible matrix $U^{(2)}$.
Again by invertibility there is a neighborhood $\U_2$ of $x$ and an element in the second row of $U^{(2)}$ that is nowhere zero on $\U_2$.
Use this element to perform Gaussian elimination on $\U_1 \cap \U_2$.
Inductively and by using the last step of \cite[Lemma 3.3]{KakKat12} we finally obtain a permutation $\pi$ on $d$ symbols and $d$ elements, say $v_{ii}$, such that
\[
\diag_{\si_\pi|_\U}(f) \cdot \diag(v_{ii}|_\U)_{i=1}^d = \diag(v_{ii}|_\U)_{i=1}^d \cdot \diag_{\rho|_\U}(f),
\]
where $\U = \U_1 \cap \U_2 \cap \dots \cap \U_d$.
Furthermore $v_{ii}(x) \neq 0$ for all $x \in \U$.
Therefore the equation $f \si_{\pi(i)}|_\U \cdot v_{ii}|_\U = v_{ii}|_\U \cdot f \rho_{i}$ implies that $\si_{\pi(i)}|_\U = \rho_i|_\U$ for all $i=1, \dots, d$.
\end{proof}

Davidson and Katsoulis \cite[Theorem 3.22]{DavKat11} showed that piecewise conjugacy is an invariant for algebraic isomorphism of tensor algebras.
The converse also was provided in several cases, including the cases of $d=2,3$ where algebraic isomorphism can be replaced by the stronger isometric isomorphism \cite[Theorem 3.25]{DavKat11}.
In fact this is proven by showing that piecewise conjugacy implies unitary equivalence in these specific cases.
It remains an open problem whether or not the converse of Proposition \ref{P: un eq pc} is true in full generality.

In particular it is not known whether the Pimsner algebras of piecewise systems are $*$-isomorphic for \emph{any} multiplicity.
Nevertheless, we aim to show that they have the same KMS-theory.
We begin with the following lemma.

\begin{lemma}\label{L: same sum comm}
Suppose that $(A,\al) \equiv (X,\si)$ and $(C,\ga) \equiv (Y,\rho)$ are piecewise conjugate by the homeomorphism $\phi \colon X \to Y$.
If $\wh{\phi} \colon C \to A$ is the implemented $*$-isomorphism then
\[
\sum_{|w| =m} \tau \al_{\ol{w}} = \sum_{|w| =m} \tau \wh{\phi}^{-1} \ga_{\ol{w}} \wh{\phi}
\]
for all $\tau \in S(A)$ and $m \in \bZ_+$.
\end{lemma}

\begin{proof}
By substituting $\rho_i$ with $\phi^{-1} \rho_i \phi$ we may assume that $X=Y$ so that for every $x \in X$ there is a neighborhood $\U$ of $x$ and a permutation $\pi$ on $d$ symbols such that $\rho_i|_\U = \si_{\pi(i)}|_\U$ for all $i=1,\dots, d$.
In particular for any $x\in X$ we obtain
\[
\{\si_i(x) \mid i=1, \dots, d \} = \{\rho_i(x) \mid i=1, \dots, d\}.
\]
For the inductive step suppose that for any $x \in X$ we have that
\[
\{\si_w(x) \mid w \in \bF_+^d, |w| = n \} = \{\rho_w(x) \mid  w \in \bF_+^d, |w| = n\}
\]
holds for all $n \leq m$.
Let a word $\mu \in \bF_+^d$ of length $m+1$ such that $\mu = i_0 w$ with $|w| = m$.
For the point $y = \si_w(x)$ we have that
\[
\si_\mu(x) = \si_{i_0}(y) \in \{\rho_i(y) \mid i=1, \dots, d\}.
\]
On the other hand by the inductive hypothesis we have that
\[
y = \si_w(x) \in \{\rho_{w'}(x) \mid w' \in \bF_+^d, |w'| = m\},
\]
therefore
\[
\rho_{i_0}(y) = \rho_{i_0} \si_w(x) = \rho_{i_0} \rho_{w'}(x) \in \{\rho_\nu(x) \mid \nu \in \bF_+^d, |\nu| = m+1\}.
\]
Thus we obtain that
\[
\si_\mu(x) \in \{\rho_\nu(x) \mid \nu \in \bF_+^d, |\nu| = m+1\}.
\]
Since $\mu$ was arbitrary and by symmetry we have that
\[
\{\si_w(x) \mid w \in \bF_+^d, |w| = m+1 \} = \{\rho_w(x) \mid w \in \bF_+^d, |w| = m+1\},
\]
for any $x \in X$.
Consequently we obtain that
\[
\{\si_{w}(x) \mid w \in \bF_+^d, |w| = m \} = \{\rho_{w}(x) \mid w \in \bF_+^d, |w| = m\},
\]
for all $m \in \bZ_+$.

Let $\tau$ be a state on $A$ that is a finite convex combination of pure states, i.e. $\tau = \sum_{k=1}^N \la_k \ev_{x_k}$ with $\sum_{i=1}^N \la_k = 1$.
Then we compute
\begin{align*}
\sum_{|w| =m} \tau \al_{\ol{w}}
& =
\sum_{k=1}^N \la_k \sum_{|w| = m} \ev_{x_k}\al_{\ol{w}}
 =
\sum_{k=1}^N \la_k \sum_{|w| = m} \ev_{\si_{w}(x_k)} \\
& =
\sum_{k=1}^N \la_k \sum_{|w| = m} \ev_{\tau_{w}(x_k)}
 =
\sum_{|w| =m} \tau \ga_{\ol{w}}
\end{align*}
for all $m \in \bZ_+$.
The proof is completed by taking limits of such states $\tau$.
\end{proof}

\begin{corollary}\label{C: comm same KMS}
Suppose that $(A,\al) \equiv (X,\si)$ and $(C,\ga) \equiv (Y,\rho)$ are piecewise conjugate by the homeomorphism $\phi \colon X \to Y$ and let $\wh{\phi} \colon C \to A$ be the induced $*$-isomorphism.
If $\be > \log d$, then the affine weak*-homeomorphism
\[
S(A) \to S(C): \tau \mapsto \tau \wh{\phi}
\]
lifts to an affine weak*-homeomorphism between the $(\si,\be)$-KMS states on $\T(A,\al)$ and on $\T(C,\ga)$.

In particular if $\be > \log d$ and $\tau \mapsto \psi_\tau$ is the parametrization obtained by Theorem \ref{T: KMS TP} then the diagram
\[
\xymatrix@C=4em@R=2em{
\tau \ar@{|->}[r] \ar@{|->}[d] & \tau \wh{\phi} \ar@{|->}[d] \\
 \psi_\tau \ar@{|-->}[r] & \psi_{\tau \wh{\phi}}
}
\]
is commutative in the sense that
\[
\psi_{\tau \wh{\phi}}(V_\mu c V_\nu^*) = \psi_{\tau}(V_\mu \wh{\phi}(c) V_\nu^*), \foral c \in C \text{ and } \nu, \mu \in \bF_+^d.
\]

The same holds for the parametrization of Proposition \ref{P: affine log d} for $\be = \log d >0$, and for the parametrization of Theorem \ref{T: KMS CP} on the KMS states on $\O(A,\al)$ and $\O(C,\ga)$.
\end{corollary}

\begin{proof}
Without loss of generality we may assume that $A = C$ and $\wh{\phi} = \id$ by substituting every $\rho_i$ by $\phi^{-1} \rho_i \phi$.
Indeed this yields a unitary equivalence between $(Y,\rho)$ and $(X, \phi^{-1}\rho \phi)$ which in turn implies that the Pimsner algebras are $*$-isomorphic.
Note that the $*$-isomorphism is given by
\[
V_\mu c V_\nu^* \mapsto V_\mu \wh{\phi}(c) V_\nu^*, \foral c \in C \text{ and } \mu, \nu \in \bF_+^d,
\]
and respects the statement.

Furthermore it suffices to prove the claims for the Toeplitz-Pimsner algebras.
Indeed by piecewise conjugacy we have that
\[
\si_i(x) \in \cup_{i=1}^d \rho_i(x)
\]
for all $x \in X$ and $i=1, \dots, d$, therefore
\[
\ol{\cup_{i=1}^d \si_i(X)} = \ol{\cup_{i=1}^d \rho_i(X)}.
\]
As a consequence we have that $J_{(A,\al)} = J_{(A,\ga)}$.

For $d=1$ piecewise conjugacy coincides with unitary equivalence and so the Pimsner algebras coincide.
When $d>1$ then the analysis on the ground states and the KMS${}_\infty$ states implies the required.
In particular notice that the ground states and the KMS${}_\infty$ states coincide.

Let $\be \geq \log d > 0$ and let $\tau \in S(A)$.
Under the simplifications we have to show that the induced $(\si,\be)$-KMS state on $\T(A,\al)$ of Proposition \ref{P: constr KMS} coincides with the induced $(\si,\be)$-KMS state on $\T(A,\ga)$ of Proposition \ref{P: constr KMS}.
To this end let $\psi_{\tau,\al}$ be the state on $\T(A,\al)$, and let $\psi_{\tau,\ga}$ be the state on $\T(A,\ga)$.
Recall that by Lemma \ref{L: same sum comm} we have that
\[
\sum_{|w| = m} \tau \al_{\ol{w}}(a) = \sum_{|w| = m} \tau\ga_{\ol{w}}(a), \foral a \in A.
\]
We then obtain
\begin{align*}
\psi_{\tau,\al}(V_\mu a V_\nu^*)
& =
\de_{\mu, \nu} \cdot (1 - e^{-\be}d) \cdot
\sum_{m=0}^\infty e^{-(m + |\mu|)\be} \sum_{|w| = m} \tau\al_{\ol{w}}(a) \\
& =
\de_{\mu, \nu} \cdot (1 - e^{-\be}d) \cdot
\sum_{m=0}^\infty e^{-(m + |\mu|)\be} \sum_{|w| = m} \tau\ga_{\ol{w}}(a) \\
& =
\psi_{\tau,\ga}(V_\mu a V_\nu^*),
\end{align*}
for all $V_\mu a V_\nu^* \in \T(A,\al)$.
Similarly, if $\tau \in AVT(A,\al)$ then
\[
\tau = \frac{1}{d} \sum_{i=1}^d \tau \al_i = \frac{1}{d} \sum_{i=1}^d \tau \ga_i,
\]
hence $\tau \in AVT(A,\ga)$ which completes the proof.
\end{proof}

\begin{remark}
It is evident that any equivalence relation between dynamical systems (even non-classical) that implies an equation as that of Lemma \ref{L: same sum comm} automatically produces a result similar to Corollary \ref{C: comm same KMS}.
In particular one just needs to check the equation of Lemma \ref{L: same sum comm} just for pure states.
The proof follows in the same way as above and it is left to the reader.
\end{remark}

\begin{acknow}
The author would like to thank the anonymous referee for the helpful remarks and comments.

The author would like to thank Guy Salomon and Natali Svirsky for their warm hospitality in Be'er Sheva.
Thankfully, the best part of one's life consists of his friendships.
\end{acknow}


\end{document}